\documentclass{article}

\usepackage[utf8]{inputenc}
\usepackage{color}
\usepackage{amssymb,amsmath,amscd,dsfont}
\newcommand{\op}[1]{\ensuremath{\operatorname{#1}}}
\newcommand{\wt}[1]{\ensuremath{\widetilde{#1}}}
\newcommand{\wh}[1]{\ensuremath{\widehat{#1}}}
\newcommand{\ol}[1]{\ensuremath{\overline{#1}}}
\newcommand{\cK}{\ensuremath{\mathcal{K}}}
\newcommand{\cZ}{\ensuremath{\mathcal{Z}}}
\newcommand{\fz}{\ensuremath{\mathfrak{z}}}
\newcommand{\fk}{\ensuremath{\mathfrak{k}}}
\newcommand{\fK}{\ensuremath{\mathfrak{K}}}

\newcommand{\fh}{\ensuremath{\mathfrak{h}}}
\newcommand{\cF}{\ensuremath{\mathcal{F}}}
\newcommand{\cS}{\ensuremath{\mathcal{S}}}
\newcommand{\mf}[1]{\ensuremath{\mathfrak{#1}}}
\newcommand{\dd}{\op{\tt{d}}}
\newcommand{\bD}{\ensuremath{\mathbb{D}}}
\newcommand{\R}{\ensuremath{\mathbb{R}}}
\newcommand{\Z}{\ensuremath{\mathbb{Z}}}
\newcommand{\Ad}{\ensuremath{\operatorname{Ad}}}
\newcommand{\ad}{\ensuremath{\operatorname{ad}}}
\newcommand{\supp}{\ensuremath{\operatorname{supp}}}
\newcommand{\Aut}{\ensuremath{\operatorname{Aut}}}
\newcommand{\GL}{\ensuremath{\operatorname{GL}}}
\newcommand{\der}{\ensuremath{\operatorname{der}}}
\newcommand{\Hom}{\ensuremath{\operatorname{Hom}}}

\newcommand{\se}{\ensuremath{\nobreak\subseteq\nobreak}}
\newcommand{\from}{\ensuremath{\nobreak\colon\nobreak}}
\renewcommand{\to}{\ensuremath{\nobreak\rightarrow\nobreak}}
\renewcommand{\phi}{\varphi}
\newcommand{\ld}{\ensuremath{\operatorname{\delta}}}
\usepackage{amssymb,amsmath,amscd}

\usepackage[amsmath,thmmarks]{ntheorem}
\theoremseparator{.}
\theoremsymbol{\rule{1ex}{1ex}}
\theorembodyfont{\upshape}
\newtheorem{definition}{Definition}[section]
\newtheorem{remark}[definition]{Remark}

\newtheorem*{proof}{Proof}

\theorembodyfont{\itshape}

\theoremsymbol{}
\newtheorem{lemma}[definition]{Lemma}
\newtheorem{proposition}[definition]{Proposition}
\newtheorem{theorem}[definition]{Theorem}
\newtheorem{corollary}[definition]{Corollary}
\newenvironment{tabsection}{}{}

\usepackage[a4paper,hscale=0.7,vscale=0.82,hmarginratio=1:1,includehead,headheight=14pt]{geometry}
\usepackage{fancyhdr}
\fancypagestyle{plain}{\fancyhead{}\fancyfoot[C]{\thepage}}

\usepackage[unicode]{hyperref}

\hypersetup{
    colorlinks = true,
	linkcolor = blue,
	citecolor = blue,
    linktocpage = true,
	pdftitle = {Universal Central Extensions of Gauge Algebras and Groups},
    pdfauthor = {Bas Janssens, Christoph Wockel},
    bookmarksopen = true,
    bookmarksopenlevel = 1,
	unicode = true,
    hypertexnames = false
  }

\usepackage[all,cmtip,line]{xy}
\CompileMatrices

\begin{document}

\title{Universal Central Extensions of Gauge Algebras and Groups}
\author{Bas Janssens\\[.5em]{\small University of Utrecht}\\{\small\texttt{mail@bjadres.nl}}\and
 		Christoph Wockel\\[.5em]{\small University of Hamburg}\\{\small\texttt{christoph@wockel.eu}}}
\maketitle

\begin{abstract}
 We show that the canonical central extension of the group of sections
 of a Lie group bundle over a compact manifold, constructed in 
\cite{NeebWockel07Central-extensions-of-groups-of-sections}, is universal.
 In doing so, we prove universality of the corresponding central
 extension of Lie algebras in a slightly more general setting.
\end{abstract}

\section{Setting of the Problem}

\begin{tabsection}
 Let $\fK \rightarrow M$ be a finite-dimensional, locally trivial bundle
 of Lie algebras. A cocycle on its Lie algebra of sections can then be
 constructed as follows. For any Lie algebra $\fk$, the derivations
 $\mathrm{d}\in\der(\fk)$ act naturally on its second symmetric tensor power
 $S^{2}(\fk)$ by
 $\mathrm{d} \cdot (x \otimes_{s} y) = \mathrm{d}(x) \otimes_s y + x \otimes_{s} \mathrm{d}(y)$,
 and we denote the quotient by
 $V(\fk) := S^{2}(\fk) /\langle\der(\fk) \cdot S^{2}(\fk)\rangle$.
 The symmetric bilinear form $\kappa \from \fk \times \fk \to V(\fk)$
 given by
 $$ \kappa(x,y) := [x \otimes_{s}y] $$ satisfies
 $\kappa(\mathrm{d}(x),y) + \kappa(x,\mathrm{d}(y)) = 0$ for all
 $\mathrm{d} \in \der(\fk)$ and is universal with this property. Now let
 $\nabla$ be a Lie connection on $\fK$, i.e.\ one that satisfies
 $\nabla [\xi,\eta] = [\nabla \xi , \eta] + [\xi, \nabla \eta]$ for all
 sections $\xi, \eta \in \Gamma(\fK)$. Then $\nabla$ induces a flat
 connection $\dd_{\nabla}$ on $V(\fK)$ by
 $\dd_{\nabla}[\xi \otimes_{s}\eta] = [\nabla\xi \otimes_{s}\eta] + [\xi \otimes_{s}\nabla\eta]$,
 where $V(\fK)$ is the vector bundle one obtains by applying
 $\fk\mapsto V(\fk)$ fibrewise. The connection $\dd_{\nabla}$ does not
 depend on $\nabla$, as any two Lie connections differ by a pointwise
 derivation, which acts trivially on $V(\fK)$. We therefore omit the
 subscript and simply write $\dd$. Using the identity
 $\dd \kappa(\xi,\eta) = \kappa(\nabla\xi,\eta) + \kappa(\xi,\nabla\eta)$
 and the compatibility of $\nabla$ with the Lie bracket, it is not hard
 to check that
 \begin{equation}\label{eqn:canonical-cocycle}
  \omega_{\nabla}\from \Gamma_{c}(\fK)\times \Gamma_{c}(\fK)\to
  \ol{\Omega}{}^{1}_{c}(M,V(\fK)),
  \quad (\eta,\xi)\mapsto 
  [\kappa(\eta,\nabla \xi)]
 \end{equation}
 defines a Lie algebra cocycle, where $\ol{\Omega}{}^{1}_{c}(M,V(\fK))$
 denotes $\Omega{}^{1}_{c}(M,V(\fK)) / \dd \Omega^{0}_{c}(M,V(\fK))$,
 and the subscript $c$ denotes compact support.

 The classes $[\omega_{\nabla}]$ in
 $H^{2}(\Gamma_{c}(\fK),\ol{\Omega}{}^{1}_{c}(M,V(\fK)))$ are naturally
 parameterised by the quotient
 $\Omega^{1}(M,\der(\fK)) / \Omega^{1}(M,\fK)$, as the Lie connections
 constitute an affine space over $\Omega^{1}(M,\der(\fK))$, and
 $\nabla - \nabla'= \ad_{\Xi}$ for some $\Xi \in \Omega^{1}(M,\fK)$
 implies that
 $(\omega_{\nabla} - \omega_{\nabla'})(\xi,\eta) = [\kappa(\xi,[\Xi,\eta])] = - [\kappa(\Xi,[\xi,\eta])]$
 is a coboundary. In particular, the class
 $[\omega_{\nabla}] = [\omega]$ is canonical if the fibres are
 semi-simple, because in that case $\der(\fK) = \fK$. We endow our
 spaces of smooth forms and sections with the usual LF-topology (cf.\
 \cite{Maier02Central-extensions-of-topological-current-algebras}), and
 denote the continuous linear maps and cohomologies by $\Hom_{\op{ct}}$
 and $H_{\op{ct}}^{*}$ respectively. The first result of this paper is
 now the following
 \begin{proposition}\label{prop:universal_algebra_cocycle}
  If $\fK$ is a finite-dimensional Lie algebra bundle with semi-simple
  fibres, then
  $[\omega]\in H_{\op{ct}}^{2}(\Gamma_{c}(\fK),\ol{\Omega}{}^{1}_{c}(M,V(\fK)))$
  is (weakly) universal, i.e.\ the map
  \begin{equation}\label{eqn:universal_algebra_cocycle}
   \Hom_{\op{ct}}(\ol{\Omega}{}^{1}_{c}(M,V(\fK)),X)\to 
   H_{\op{ct}}^{2}(\Gamma_{c}(\fK),X), \quad
   \varphi \mapsto [\varphi \op{\circ} \omega]   
  \end{equation}
  is bijective for each topological vector space $X$, considered as a
  trivial $\Gamma_{c}(\fK)$-module.
 \end{proposition}
 Note that by Proposition \ref{prop:supportDecreasing} and \cite[Lem.\
 1.12]{Neeb02Universal-central-extensions-of-Lie-groups} the associated
 central extension of locally convex Lie algebras  is also universal in
 the stronger sense of
 \cite{Neeb02Universal-central-extensions-of-Lie-groups}.

 The proof of the previous proposition will be carried out in Section
 \ref{sect:universality_of_the_lie_algebra_cocycle}. Note that for
 semi-simple Lie algebras, $\kappa$ is the universal symmetric invariant
 bilinear form \cite[App.\
 B]{NeebWockel07Central-extensions-of-groups-of-sections}. It is equal
 to the Killing form for all central simple real Lie algebras, but not
 for e.g.\ the real simple Lie algebras $\mathfrak{sl}_{n}(\mathbb{C})$.
 The motivating example for the above proposition is the gauge algebra
 $\Gamma(\ad(P))$ of a principal fibre bundle $P\to M$, whence the title
 of this paper.

 We now formulate the corresponding result for Lie groups rather than
 Lie algebras. Let $\cK \to M$ be a finite-dimensional, locally trivial
 bundle of Lie groups, and set $\fK := L(\cK)$. Unlike in the case of
 Lie algebras, we will assume the base manifold $M$ to be compact and
 connected. Consequently, the fibres of $\cK$ will all be isomorphic to
 a single Lie group $K$ with Lie algebra $\fk$. We will assume that
 $\fk$  is semi-simple, and that $\pi_0(K)$ is finitely generated.
 The latter implies that $\Aut(K)$ carries a natural Lie group structure
 \cite[Ch.\ III, \S
 10]{Bourbaki98Lie-groups-and-Lie-algebras.-Chapters-1--3} modelled on
 $\fk = \der(\fk)$. We may thus consider $\cK$ to be associated to its
 principal $\Aut(K)$-frame bundle
 $\op{Fr}(\cK):=\cup_{x\in M}\op{Iso}(\cK_{x},K)$, placing us in the
 setting considered in
 \cite{NeebWockel07Central-extensions-of-groups-of-sections}. The group
 $\Gamma(\cK)$ of smooth sections then carries naturally the structure
 of a Fr\'echet-Lie group with Lie algebra $\Gamma(\fK)$, see \cite[App.
 A]{NeebWockel07Central-extensions-of-groups-of-sections}. The connected
 component $\Gamma(\cK)_{0}$ possesses a central extension, which we
 will show to be universal. It is constructed as follows.

 First we have to ensure a technical condition in order to use the more
 detailed results from
 \cite{NeebWockel07Central-extensions-of-groups-of-sections}. For this
 we consider the homomorphism $\rho\from \pi_{1}(M)\to \GL(V(\fk))$,
 given by composing the connecting homomorphism
 $\delta\from \pi_{1}(M)\to \pi_{0}(\Aut(K))$ of the fibration
 $\op{Fr}(\cK)\to M$ with the natural representation
 $\pi_{0}(\Aut(K))\to \GL(V(\fk))$. ($\Aut(K)$ acts naturally on
 $V(\fk)$, and $\Aut(K)_{0}$ acts trivially because $\der(\fk)$ does
 so.)

 If now $\bD_{\cK}:= \rho(\pi_{1}(M))$ is finite, then by \cite[Cor.\
 4.18]{NeebWockel07Central-extensions-of-groups-of-sections} the period
 group $\Pi_{\omega}$ associated to $\omega$ is discrete, and by
 \cite[Th.\
 7.9]{Neeb02Central-extensions-of-infinite-dimensional-Lie-groups} there
 exists a central extension
 \begin{equation*}
  \ol{\Omega}{}^{1}(M,V(\fK))/\Pi_{\omega}\to\wh{\Gamma(\cK) _{0}}\xrightarrow{{q}} \wt{\Gamma(\cK) _{0}}
 \end{equation*}
 (i.e.\  a short exact sequence of locally convex Lie groups with
 central kernel that is a smooth principal bundle, cf.\
 \cite{Neeb02Central-extensions-of-infinite-dimensional-Lie-groups})
 where $c\from\wt{\Gamma(\cK)_{0}}\to {\Gamma(\cK)_{0}}$ denotes the
 simply connected cover. In order so see that $c \op{\circ} q$ also
 defines a central extension we observe that $\Gamma(\cK)_{0}$ is a
 covering group of $\Gamma(\Ad(\op{Fr}(\cK)))_{0}$. In fact, the exact
 sequence $Z(K)\to K\to \op{Inn}(K)\se\Aut(K)$ induces a (fibrewise)
 exact sequence
 \begin{equation*}
  \cZ(\cK) \to \cK \to \mathcal{I}nn(\cK),
 \end{equation*}
 which leads to a covering
 $\Gamma(\cZ(\cK))\to \Gamma(\cK)_{0}\to \Gamma(\mathcal{I}nn(\cK))_{0}$
 ($\Gamma(\cZ(\cK))$  is discrete since $Z(K)$ is so). Note that
 $\mathcal{I}nn(\cK)$ is obtained from $\cK$ by a push-forward along the
 morphism $K\to \op{Inn}(K)$ of $\Aut(K)$-spaces. Since $K$ is
 semi-simple, we have that $\op{Inn}(K)$ is an open subgroup of
 $\Aut(K)$, and thus
 $\Gamma(\Ad(\op{Fr}(\cK)))_{0}=\Gamma(\mathcal{I}nn(\cK))_{0}$. It has
 been shown in \cite[Cor.\
 22]{NeebWockel07Central-extensions-of-groups-of-sections} that the
 adjoint action of $\Gamma(\ad(\op{Fr}(\cK)))$ integrates to an action
 of $\Gamma(\Ad(\op{Fr}(\cK)))_{0}$, and since
 $\Gamma(\cK)_{0}\to \Gamma(\Ad(\op{Fr}(\cK)))_{0}$ is a covering, the
 adjoint action of $\Gamma(\fK)\cong \Gamma(\ad(\op{Fr}(\cK)))$
 integrates to an action of $\Gamma(\cK)_{0}$ on
 $\ol{\Omega}{}^{1}(M,V(\fK))\oplus_{\omega}\Gamma(\fK)$. From
 \cite[Rem.\
 7.14]{Neeb02Central-extensions-of-infinite-dimensional-Lie-groups} it
 now follows that
 \begin{equation}\label{eqn:centrExtGaugeGroup}
  Z\to\wh{\Gamma(\cK) _{0}}
  \xrightarrow{ c \op{\circ} q} {\Gamma(\cK) _{0}}
 \end{equation}
 with
 $Z:=\ker(c \op{\circ} q)\cong (\ol{\Omega}^{1}_{c}(M,V(\fK))/\Pi_{\omega})\times \pi_{1}(\Gamma(\cK)_{0})$
 is a central extension, which turns out to be universal:

 \begin{theorem}\label{thm:universal_group_extension}
  Let $\cK \to M$ be a finite-dimensional Lie group bundle over a
  compact and connected manifold $M$, such that its typical fibre $K$ is
  semi-simple and has finitely generated $\pi_{0}(K)$. If, moreover,
  $\bD_{\cK}$ is finite, then the central extension
  \eqref{eqn:centrExtGaugeGroup} is universal for abelian Lie groups
  modelled on Mackey-complete locally convex spaces.
 \end{theorem}
 This means that for each central extension
 $A\to \wh{G}\to \Gamma(\cK) _{0}$ of locally convex Lie groups that is
 a smooth principal bundle and such that $L(A)$ is Mackey-complete,
 there exists a unique morphism
 $\varphi\from \wh{\Gamma(\cK)}_{0} \to \wh{G}$ such that the diagram
 \begin{equation*}
  \vcenter{\xymatrix{
  Z \ar[r]\ar[d]^{\left.\varphi\right|_{Z}} & \wh{\Gamma(\cK)}_{0} \ar[r]\ar[d]^{\varphi} &
  \Gamma(\cK) _{0}\ar@{=}[d]\\
  A \ar[r] & \wh{G} \ar[r] & \Gamma(\cK)_{0}
  }}
 \end{equation*}
 commutes.
\end{tabsection}

\begin{tabsection}
 In general it is not easy to decide whether the group $\bD_{\cK}$ is finite.
 In order to make the above theorem applicable, we list here some conditions
 that ensure the finiteness of $\bD_{\cK}$.
 \begin{itemize}
  \item If the image of the connecting homomorphism
        $\delta\from \pi_{1}(M)\to \pi_{0}(\Aut(K))$ of the fibration
        $\op{Fr}(\cK)\to M$ is finite, then $\bD_{\cK}$ is finite. This is in
        particular the case if $\cK$ is trivial, reproducing the universal
        central extension of $\Gamma(\cK)_{0}\cong C^{\infty}(M,K)_{0}$ of
        \cite{MaierNeeb03Central-extensions-of-current-groups}.
  \item If $\fk$ is compact and simple (or, more generally, central simple
        real), then the Killing form is universal and $\Aut(K)$-invariant. Thus
        the homomorphism $\pi_{0}(\Aut(K))\to \GL(V(\fk))$ vanishes, so that
        $\bD_{\cK}$ and $V(\fK)$ are trivial.
  \item If $\pi_{0}(K)$ and $Z(K_0)$ are both finite, then so is
        $\pi_{0}(\Aut(K))$, and therefore $\bD_{\cK}$. This can be seen as
        follows. The image of the natural homomorphism
        $\Aut(K) \rightarrow \Aut(\fk) \times \Aut(\pi_0(K))$ is both open and
        closed, and its kernel is naturally identified with the group of
        crossed homomorphisms from $\pi_0(K)$ to $Z(K_0)$. Since both the
        kernel and the image have finitely many connected components (the group
        $\pi_{0}(\Aut(\fk))$ is finite, see
        \cite[Thm.~4]{Whitney57Elementary-structure-of-real-algebraic-varieties}
        or \cite{Gundogan10The-component-group-of-the-automorphism-group-of-a-simple-Lie-algebra-and-the-splitting-of-the-corresponding-short-exact-sequence}),
        $\Aut(K)$ must have finitely many components as well.
 \end{itemize}
 
 Note that the concepts of universality and weak universality of cocycles and
 central extensions that are used in our main references
 \cite{Neeb02Universal-central-extensions-of-Lie-groups} and
 \cite{Maier02Central-extensions-of-topological-current-algebras} differ
 slightly, but coincide in the case of \emph{perfect} Lie algebras. We will see
 in the next section that $\Gamma(\fK)$ is in fact perfect. We will give
 precise references for the equality of these concepts at each stage where an
 ambiguity might occur.
 
 Our results prove the universality of the above central extension of gauge
 algebras which has already been claimed at some places in the literature, see
 for instance
 \cite{LosevMooreNekrasovShatashvili98Central-extensions-of-gauge-groups-revisited}
 or \cite{MohrdieckWendt04Integral-conjugacy-classes-of-compact-Lie-groups}.
 Note also that a common mistake is made in some treatments of this subject by
 considering the Killing form as universal invariant bilinear form $\kappa$,
 which is not always justified (see above).
\end{tabsection}

\section{Universality of the Lie Algebra Cocycle}
\label{sect:universality_of_the_lie_algebra_cocycle}

\begin{tabsection}
 This section is devoted to the proof of Proposition
 \ref{prop:universal_algebra_cocycle}, in which a pivotal role will be
 played by the fact that our Lie algebra cocycles are essentially local in
 nature. We will state and prove this in a slightly more general setting
 where the fibres of our Lie algebra bundle are allowed to be
 infinite-dimensional. More precisely, we will require $\fK \to M$ to be
 a locally trivial bundle of locally convex topological Lie algebras,
 with a base $M$ that is finite-dimensional but not necessarily compact.
 We then equip its space of compactly supported sections
 $\Gamma_{c}(\fK) = {\displaystyle \lim_{\longrightarrow}}\, \Gamma_{K}(\fK)$
 with the usual inductive limit topology.

 Throughout this section, $X$ will denote an arbitrary topological
 vector space, considered as a trivial module for the Lie algebra in
 question.
\end{tabsection}

\begin{definition}
 A (not necessarily continuous) 2-cocycle on $\Gamma_{c}(\fK)$ with
 values in $X$ is called \emph{diagonal} if $\psi(\eta,\xi)=0$ whenever
 $\supp(\eta)\cap \supp(\xi)=\emptyset$.
\end{definition}

\begin{tabsection}
 Combined with Proposition \ref{prop:supportDecreasing}, the following
 Lemma shows that all continuous 2-cocycles are diagonal if the fibres
 of $\fK$ are topologically perfect, and that all 2-cocycles (also the
 non-continuous ones) on $\Gamma_{c}(\fK)$ are diagonal if the fibres of
 $\fK$ are finite-dimensional and perfect.
\end{tabsection}

\begin{lemma}\label{lem:perfectness}
 If $\Gamma_{c}(\left.\fK\right|_{U})$ is topologically perfect for each
 open $U\se M$, then every continuous cocycle is diagonal. If, moreover,
 each $\Gamma_{c}(\left.\fK\right|_{U})$ is perfect, then every cocycle
 is diagonal.
\end{lemma}

\begin{proof}
 Suppose that $\xi$ and $\eta$ in $\Gamma_{c}(\fK)$ have disjoint
 support and set $U:=M\backslash \supp(\xi)$. Since $\xi$ and $\eta$
 have disjoint support, we have that
 $\left.\eta\right|_{U}\in \Gamma_{c}(\left.\fK\right|_{U})$. By
 assumption, we can write
 \begin{equation*}
  \left.\eta\right|_{U}=\lim _{\stackrel{\xrightarrow{~~}}{i}} \eta_{i},
 \end{equation*}
 where $(\eta_{i})_{i\in I}$ is a convergent net and
 $\eta_{i}=\sum_{j}[\mu_{j,i},\nu_{j,i}]$ is a finite sum of commutators
 in $\Gamma_{c}(\left.\fK\right|_{U})$. In case that
 $\Gamma_{c}(\left.\fK\right|_{U})$ is perfect, we may assume that $I$
 is finite.

 We now set $\mu'_{j,i}$ to be the continuous extension of $\mu_{j,i}$ by zero,
 and likewise we define $\nu'_{j,i}$ by extending $\nu_{j,i}$.
 Observe that we
 have in particular $[\xi,\mu'_{j,i}]=[\xi,\nu'_{j,i}]=0$. This now
 implies
 \begin{multline*}
  \psi(\xi,\eta)=
  \psi(\xi,\lim_{\stackrel{\xrightarrow{~~}}{i}}\sum_{j}[\mu_{j,i}',\nu_{j,i}'])
  =\lim_{\stackrel{\xrightarrow{~~}}{i}}\sum_{j}\psi(\xi,[\mu_{j,i}',\nu_{j,i}'])\\
  =\lim_{\stackrel{\xrightarrow{~~}}{i}}\sum_{j}\big(\psi(\mu_{j,i}',[\xi,\nu'_{j,i}])+
  \psi(\nu'_{j,i},[\mu'_{j,i},\xi])\big)
  =\lim_{\stackrel{\xrightarrow{~~}}{i}}\sum_{j}\big(\psi(\mu'_{j,i},0)+
  \psi(\nu'_{j,i},0)\big)=0
 \end{multline*}
 for each cocycle in the case of finite $I$ and for each continuous
 cocycle in the case of arbitrary $I$.
\end{proof}

\begin{remark}
 The previous proof also works for the Lie algebra of compactly supported
 vector fields $\op{Vec}_{c}(M)$ (cf.\ \cite[Cor.\
 7.4]{Janssens10Transformation--Uncertainty.-Some-Thoughts-on-Quantum-Probability-Theory-Quantum-Statistics-and-Natural-Bundles} and
 \cite{ShanksPursell54The-Lie-algebra-of-a-smooth-manifold}), and can even be
 generalised (cf.\
 \cite{Amemiya75Lie-algebra-of-vector-fields-and-complex-structure}) to
 $\op{Vec}(M)$. Now both Lie algebras satisfy the conditions from Lemma
 \ref{lem:perfectness}, so that their second Lie algebra cohomology is diagonal
 (cf.\ \cite[Cor.\
 6.3]{GelfandFuks70Cohomologies-of-the-Lie-algebra-of-formal-vector-fields}).
 Results such as Peetre's theorem
 \cite{Peetre60Rectification-a-larticle-Une-caracterisation-abstraite-des-operateurs-differentiels}
 tell one that continuity and diagonality are not worlds apart, which raises
 the interesting question of whether $H^{2}(\op{Vec}(M),\mathbb{R})$ is in fact
 isomorphic to the \emph{continuous} cohomology
 $H^{2}_{\op{ct}}(\op{Vec}(M),\mathbb{R})$. The latter has been calculated
 explicitly (cf.\
 \cite{GelfandFuks70Cohomologies-of-the-Lie-algebra-of-formal-vector-fields}
 and \cite{Fuks86Cohomology-of-infinite-dimensional-Lie-algebras}).
\end{remark}

\begin{proposition}\label{prop:supportDecreasing}
 If the fibres of $\fK$ are topologically perfect, then so is
 $\Gamma_{c}(\fK)$. If the fibres of $\fK$ are finite-dimensional
 perfect Lie algebras, then $\Gamma_{c}(\fK)$ is even perfect. The same
 conclusions hold in particular for
 $\Gamma_{c}(\left.\fK\right|_{U})$ with arbitrary open $U\se M$.
\end{proposition}

\begin{proof}
 Since each $\xi\in \Gamma_{c}(\fK)$ is compactly supported, it can be
 written as $\xi=\sum_{i=1}^{N}\xi_{i}$ with each $\xi_{i}$ having
 compact support in a trivialising open subset $U_{i}$. In order to
 substantiate the first claim, it thus suffices to show that
 $C^{\infty}_{c}(U_{i},\fk)$ is topologically perfect if $\fk$ is
 topologically perfect. We clearly have
 \begin{equation*}
  [C_{c}^{\infty}(U_{i})\otimes\fk,  C_{c}^{\infty}(U_{i})\otimes\fk]=
  C^{\infty}_{c}(U_{i})\otimes[\fk,\fk].
 \end{equation*}
 Now $C^{\infty}_{c}(U_i)\otimes [\fk,\fk]$ is dense in
 $C^{\infty}_{c}(U_{i})\op{\wh{\otimes}}\ol{\fk}$ ($\ol{\fk}$ denoting
 the uniform completion of $\fk$), and since
 $C^{\infty}_{c}(U_{i})\op{\wh{\otimes}}\ol{\fk}\cong C_{c}^{\infty}(U_{i},\ol{\fk})$
 \cite[Chap.\ II, p.\
 81]{Grothendieck55Produits-tensoriels-topologiques-et-espaces-nucleaires},
 $C^{\infty}_{c}(U_i)\otimes [\fk,\fk]$ is dense, considered as a subspace of
 $C^{\infty}_{c}(U_{i},\fk)$. Thus
 $[C^{\infty}_{c}(U_{i},\fk),C^{\infty}_{c}(U_{i},\fk)]$ is dense in
 $C^{\infty}_{c}(U_{i},\fk)$.

 If $\fk$ is a finite-dimensional, perfect Lie algebra, then all spaces
 considered above are in fact equal and we have
 \begin{equation*}
  [C^{\infty}_{c}(U_{i},\fk),C^{\infty}_{c}(U_{i},\fk)]= 
  [C_{c}^{\infty}(U_{i})\otimes\fk,  C_{c}^{\infty}(U_{i})\otimes\fk]=
  C^{\infty}_{c}(U_{i})\otimes[\fk,\fk]=C^{\infty}_{c}(U_{i},\fk).
 \end{equation*}
\end{proof}

For the next result, recall that a monopresheaf is a presheaf that satisfies the 
``local identity'' axiom but not necessarily the ``gluing'' axiom.

\begin{corollary}\label{cor:monopresheaf}
 If the fibres of $\fK$ are topologically perfect, then
 $\cS_{\op{ct}}(U) = H^{2}_{\op{ct}}(\Gamma_{c}(\fK|_{U}) , X)$
 constitutes a monopresheaf of vector spaces. If the fibres are
 finite-dimensional perfect Lie algebras, then the same applies to
 $\cS(U) = H^{2}(\Gamma_{c}(\fK|_{U}) , X)$.
\end{corollary}
\begin{proof}
 If $V \subseteq U$, then the
 restriction map $\rho_{VU}\from\cS(U)\to\cS(V)$ is defined as
 $\rho_{VU}([\psi_{U}]) = [\psi_{V}]$ with
 $\psi_{V}(\xi,\eta) := \psi_{U}(\wt{\xi},\wt{\eta})$, where $\wt{\xi}$
 and $\wt{\eta}$ denote the extensions by zero of $\xi$ and $\eta$ from
 $V$ to $U$. In particular, $\psi_V$ is continuous if $\psi_U$ is so.
 The class $[\psi_V]$ does not depend on the choice of
 $\psi_{U}\in [\psi_U]$. Indeed, if
 $\psi_{U} - \psi'_{U} = \ld \beta_{U}$, then
 $\psi_{V}(\xi,\eta) - \psi'_{V}(\xi,\eta) = \beta_{U}([\wt{\eta},\wt{\xi}]) = \beta_{U}(\wt{[\eta,\xi]}) = \ld \beta_{V}(\xi,\eta)$,
 where again $\beta_V$ is continuous if $\beta_U$ is. The presheaf
 property $\rho_{WV}\circ \rho_{VU} = \rho_{WU}$ is clear from the
 definition.

 The fact that this presheaf is actually a monopresheaf will now follow
 from Lemma \ref{lem:perfectness} and Proposition
 \ref{prop:supportDecreasing}, which tell us that the cocycles
 considered are diagonal. Let $\{V_i\}_{i\in I}$ be an open cover of
 $U$. For the `local identity' axiom to hold, we must prove that
 $[\psi_U]$ vanishes if its restriction $[\psi_{i}]$ (with
 $\psi_i := \psi_{V_i}$) vanishes for all $i \in I$. Replace
 $\{V_i\}_{i \in I}$ by a subcover that intersects every compactum in
 only finitely many $V_i$, and equip it with a partition of unity
 $\{\lambda_{i}\}_{i\in I}$. Let $\psi_{i} = \ld\beta_{i}$. We then
 define
 \begin{equation*}
  \beta_{U}(\chi) := \sum_{i \in I}\beta_{i}(\lambda_i \chi)
 \end{equation*}
 for $\chi \in \Gamma_{c}(\fK|_{U})$, the sum containing only finitely
 many nonzero terms because $\supp(\chi)$ intersects only finitely many
 $V_i$. Clearly $\beta_{U}$ is continuous if all the $\beta_i$ are. We
 prove that $\psi_{U} = \delta \beta_{U}$. By diagonality of the cocycle
 $\psi_{U}$, we can write
 $\psi_{U}(\lambda_i\xi,\eta) = \psi_{U}(\lambda_i\xi,\lambda'_{i}\eta)$,
 where $\lambda'_{i}$ is some function with support contained in $V_i$
 that satisfies $\lambda'_{i}\equiv 1$ on a neighbourhood of
 $\supp(\lambda_i)$. Indeed,
 $\psi_{U}(\lambda_i\xi, \eta - \lambda'_{i}\eta) = 0$ because
 $\lambda_i\xi$ and $\eta - \lambda'_{i}\eta$ have disjoint support. We
 can therefore write
 \begin{equation*}
  \psi_{U}(\xi,\eta) 
  =
  \sum_{i \in I}\psi_{U}(\lambda_{i}\xi , \lambda'_{i}\eta)
  = 
  \sum_{i\in I}\psi_{i}(\lambda_{i}\xi , \lambda'_{i}\eta)
  =
  \sum_{i \in I}\beta_{i}(\lambda_{i}\lambda'_{i}[\xi , \eta])
  =  
  \beta_{U}([\xi, \eta])\,,
 \end{equation*}
 where in the last step we used $\lambda_{i}\lambda'_{i} = \lambda_{i}$.
 Thus $[\psi_{U}] = [\delta \beta_{U}] = 0$, as required.
\end{proof}

\begin{proposition}\label{prop:sheaf}
 The assignment $\cF(U) = \Hom(\ol{\Omega}{}^{1}_{c}(U,V(\fK)),X)$
 constitutes a sheaf of vector spaces.
The same goes for
 $\cF_{\op{ct}}(U) = \Hom_{\op{ct}}(\ol{\Omega}{}^{1}_{c}(U,V(\fK)),X)$.
\end{proposition}
Recall that
 $\ol{\Omega}{}^{1}_{c}(U,V(\fK))$ was defined as $\Omega^1_{c}(U,V(\fK))/\dd \Omega^0_{c}(U,V(\fK))$.
For two vector spaces $X$ and $Y$, 
 $\Hom(X,Y)$ denotes the space of linear maps from $X$ to $Y$. 
If $X$ and $Y$ happen to be topological vector spaces, then
 $\Hom_{\op{ct}}(Y,X)$ is the space of continuous
 linear maps.
\begin{proof}
 Let $W \subseteq V$, with $W$ and $V$ open in $M$. The restriction
 $\rho_{WV} \from\cF(V)\to\cF(W)$ is the dual of the map
 $\iota_{VW} : \ol{\Omega}^1_c(W,V(\fK)) \rightarrow \ol{\Omega}^1_c(V,V(\fK))$
 defined as follows. Take $[\omega_{W}] \in \Omega{}^{1}_{c}(W,V(\fK))$,
 choose a representative $\omega_{W}$, extend it by zero to
 $\omega_{V} \in \Omega{}^{1}_{c}(V,V(\fK))$ and take its class
 $[\omega_V] \in \ol{\Omega}{}^{1}_{c}(V,V(\fK))$. The result does not
 depend on the choice of representative. Indeed, if
 $\omega_{W} - \omega'_{W} = \dd \gamma_{W}$ with
 $\gamma_{W} \in \Omega^0_{c}(W,V(\fK))$, then one can extend
 $\gamma_{W}$ by zero to $\gamma_{V} \in \Omega^0_{c}(V,V(\fK))$ to find
 $\omega_V - \omega'_{V} = \dd \gamma_{V}$. We can therefore define
 $\iota_{VW}([\omega_W]) = [\omega_V]$. The fact that
 $\cF(U) := \Hom(\ol{\Omega}{}^{1}_{c}(U,V(\fK)),X)$ with
 $\rho_{WV} \from \cF(V) \to \cF(W)$ is a presheaf,
 $ \rho_{XW}\circ\rho_{WV} = \rho_{XV}$, follows from the fact that
 $\check{\cF}(U) := \ol{\Omega}{}^{1}_{c}(U,V(\fK))$ with
 $\iota_{VW} \from \check{\cF}(W)\to\check{\cF}(V)$ is a precosheaf,
 $\iota_{VW}\circ\iota_{WX} = \iota_{VX}$. An analogous statement holds
 for $\cF_{\mathrm{ct}}$ because the maps $\iota_{WV}$ are continuous.

 In order to show that $\cF$ is a sheaf, it suffices to show that
 $\check{\cF}$ is a cosheaf.
 For this, one needs to check (\cite[Prop.~1.3]{Bredon68Cosheaves-and-homology}) the following
 3 statements.
 \begin{enumerate}
        \item\label{item:1} For all open $V,W \subseteq M$, we have
        $ \iota_{V \cup W , V} \check{\cF}(V) + \iota_{V \cup W , W} \check{\cF}(W) = \check{\cF}(V \cup W)\,.$
        \item\label{item:2} If
        $\iota_{V \cup W , V}[\omega_V] = \iota_{V \cup W , W}[\omega_W]$
        for $[\omega_V] \in \check{\cF}(V)$ and
        $[\omega_W] \in \check{\cF}(W)$, then there exists an element
        $[\omega_{V\cap W}] \in \check{\cF}(V\cap W)$ such that
        $[\omega_V] = \iota_{V,V\cap W} [\omega_{V\cap W}]$ and
        $[\omega_W] = \iota_{W,V\cap W} [\omega_{V\cap W}]$.
        \item\label{item:3} For every system $\{V_i\}_{I}$ directed
        upwards by inclusion, the natural map
        ${\displaystyle \lim_{\rightarrow}} \, \check{\cF}(V_i) \to \check{\cF}(\cup_{I}V_i)$
        is an isomorphism.
 \end{enumerate}
 Statement \ref{item:1} is true because one can use a partition of unity
 $\{\lambda_V , \lambda_W\}$ to write
 $\omega_{V\cup W} = \lambda_{V}\omega_{V\cup W} + \lambda_{W}\omega_{V\cup W}$.

 We proceed to prove statement \ref{item:2}. Suppose that
 $\iota_{V \cup W , V}[\omega_V] = \iota_{V \cup W , W}[\omega_W]$. Then
 $\omega_V = \omega_W + \dd \gamma_{V\cup W}$, with
 $\supp(\gamma_{V\cup W}) \subseteq V \cup W$. Write
 $\gamma_{V\cup W} = \gamma_V - \gamma_W$, with
 $\supp(\gamma_{V}) \subseteq V$ and $\supp(\gamma_{W}) \subseteq W$.
 Then $\omega_V - \dd \gamma_V = \omega_W - \dd \gamma_W$, so that the
 support of both is contained in $V \cap W$. Statement \ref{item:2} then
 holds with
 $[\omega_{V\cap W}] = [\omega_V - \dd \gamma_V] = [\omega_W - \dd \gamma_W]$
 in $\check{\cF}(V\cap W)$.

 Finally, we verify statement \ref{item:3} by first observing that since
 the support of any $\omega \in \Omega{}^{1}_{c}(\cup_{I}V_i)$ is
 compact, it is contained in some $V_{i_{1}}\cup...\cup V_{i_{l}}$. 
 This shows surjectivity. To
 verify injectivity, we observe that if $[\omega]=0$ in
 $\check{\cF}(\cup_{I}V_i)$, then $\omega=\dd \gamma$ for
 $\gamma\in \Omega_{c}^{0}(\cup_{I}V_i,V(\fK))$. Since
 $\supp(\gamma)$ is compact, we have
 $\gamma\in \Omega_{c}^{0}(V_{i_{1}}\cup...\cup V_{i_{l}},V(\fK))$. Thus
 $[\omega]=0$ in $\check{\cF}(V_{i_{1}}\cup...\cup V_{i_{l}})$ and, consequently, in
 ${\displaystyle \lim_{\rightarrow}} \, \check{\cF}(V_i)$.
\end{proof}

\begin{tabsection}
 From now on, we restrict attention to the case where the fibres of
 $\fK$ are finite-dimensional, semi-simple Lie algebras. We have
 exhibited the sheaf
 $\cF_{\op{ct}}(U) = \Hom_{\op{ct}}(\ol{\Omega}{}^{1}_{c}(U,V(\fK)),X)$
 and the monopresheaf
 $\cS_{\op{ct}}(U) = H^{2}_{\op{ct}}(\Gamma_{c}(\fK|_{U}),X)$. The
 canonical class $[\omega]$ then induces a morphism
 \begin{equation*}
  \mu_{U} \from \cF_{\op{ct}}(U) \to \cS_{\op{ct}}(U)
 \end{equation*}
 of presheaves. For $\xi, \eta \in \Gamma_{c}(\fK|_{U})$, we have
 $\omega(\eta,\xi) = [\kappa(\eta,\nabla\xi)]$ in
 $\ol{\Omega}{}^{1}_{c}(U,V(\fK))$, and the morphism $\mu_{U}$ is then
 simply defined as $\mu_U \phi = [\phi \circ \omega]$.

 If $U \subset M$ is a trivialising neighbourhood, then a local
 trivialisation $\fK|_{U} \cong U \times \fk$ induces isomorphisms
 $\cF_{\op{ct}}(U) \cong \Hom_{\op{ct}}(\ol{\Omega}{}^{1}_{c}(U,V(\fk)),X)$
 and $\cS_{\op{ct}}(U) \cong H^{2}_{\op{ct}}(C_{c}^{\infty}(U,\fk),X)$.
 The map $\mu_{U}$ then takes the shape
 $\phi \mapsto [\phi \circ \omega_{U,\fk}]$ with $\omega_{U,\fk}$ the
 cocycle $\omega_{U,\fk}(f,g) = [\kappa(f,dg)]$. Indeed, $\nabla\xi$
 corresponds with $dg + [A,g]$ in the local trivialisation 
 ($\der(\fk) \cong \fk$ for semi-simple Lie algebras) so that
 $\kappa(\eta,\nabla\xi)$ corresponds with $\kappa(f,dg + [A,g])$. This
 differs from $\kappa(f,dg)$ by a mere coboundary $-\kappa(A,[f,g])$.
 The following theorem shows that $\mu_U$ is an isomorphism for
 sufficiently small $U$.
\end{tabsection}

\begin{theorem}\label{thm:universal-current-algebra-cocycle}
 If $\fk$ is a finite-dimensional semi-simple Lie algebra and $U$ is a
 finite-dimensional manifold, then the cocycle
 \begin{equation}\label{eqn:current-cocycle}
  \omega_{U,\fk}\from C_{c}(U,\fk)\times C_{c}^{\infty}(U,\fk)\to \Omega{}^{1}(U,V(\fk))/dC^{\infty}(U,\fk),\quad
  (f,g)\mapsto [\kappa(f,dg)]
 \end{equation}
 is
 (weakly) universal. This means that the linear map
 \begin{equation*}
  \Hom_{\op{ct}}((\Omega{}^{1}_{c}(U,V(\fk))/dC^{\infty}_{c}(U,\fk)),X)\to 
  H_{\op{ct}}^{2}(C_{c}^{\infty}(U,\fk),X),
  \quad \varphi\mapsto [\varphi \op{\circ}\omega_{U,\fk}]
 \end{equation*}
 is an isomorphism.
\end{theorem}

Note that since $C_{c}^{\infty}(M,\fk)$ is not unital we cannot use \cite[Th.\
16]{Maier02Central-extensions-of-topological-current-algebras} directly, as 
claimed in \cite[Cor.\
18]{Maier02Central-extensions-of-topological-current-algebras}.

\begin{proof}
 The combination of \cite[Th.\
 11]{Maier02Central-extensions-of-topological-current-algebras} and
 \cite[Th.\
 16]{Maier02Central-extensions-of-topological-current-algebras} shows
 that
 \begin{equation*}
  \Hom((\Omega{}^{1}_{c}(U,V(\fk))/dC^{\infty}_{c}(U,\fk)),X)\to   
  H_{\op{ct}}^{2}(C_{c}^{\infty}(U,\fk)\rtimes \fk ,X),
  \quad \varphi\mapsto [\varphi \op{\circ}\omega_{U,\fk}]
 \end{equation*}
 is an isomorphism. It remains to be shown that the canonical inclusion
 $i \from C^{\infty}_{c}(U,\fk)\to C^{\infty}_{c}(U,\fk)\rtimes\fk$
 induces an isomorphism $H^{2}_{\op{ct}}(i)$. We first note that we can
 extend each cocycle $\omega$ on $C^{\infty}_{c}(U,\fk)$ to
 $C^{\infty}_{c}(U,\fk)\rtimes\fk$ if we interpret $\fk$ as constant
 functions. In other words, we set $\omega(x,y) = 0$ for 
 $x,y \in \fk$, and
 $\omega(f,x) := \omega(f,\lambda \cdot x)$ for
 $f\in C_{c}^{\infty}(U,\fk)$, $x\in\fk$ and
 $\lambda \in C_{c}^{\infty}(U)$ with $\lambda\equiv 1$ on $\supp(f)$.
 Since $\omega$ is diagonal, this does not depend on the choice of
 $\lambda$. 
The extension is again a cocycle
$\omega(x,[f,g]) + \mathrm{cyclic} = \omega(\lambda \cdot x, [f,g]) + \mathrm{cyclic} = 0$
and 
$\omega(f,[x,y]) + \mathrm{cyclic} = 
\omega(f,[\lambda \cdot x,\lambda' \cdot y]) + \mathrm{cyclic} = 0$
for $x,y \in \fk$, $f,g \in C_{c}^{\infty}(U,\fk)$, and $\lambda$
equal to 1 on $\supp(f)$ (and $\supp(g)$), $\lambda'$ equal to 1 on
$\supp({\lambda})$.
This shows that $H^{2}_{\op{ct}}(i)$ is surjective.

 If $[\omega]\in \ker(H^{2}_{\op{ct}}(i))$, then
 $\omega(f,g)=\lambda([f,g])$ for $f,g\in C_{c}^{\infty}(U,\fk)$ and
 $\lambda\from C_{c}^{\infty}(U,\fk)\to X$ linear and continuous. If we
 extend $\lambda$ by $0$ to $\fk$, then
$\omega' = \omega - \lambda \circ [\,\cdot\, , \,\cdot\,]$ is a cocycle
that vanishes on $C_{c}^{\infty}(U,\fk) \times C_{c}^{\infty}(U,\fk)$.
We have $\omega'([f,g],x) + \omega'([g,x],f) + \omega'([x,f],g) = 0$
and as the last two terms vanish, so does the first. Since $C_{c}^{\infty}(U,\fk)$ 
is perfect, $\omega'$ vanishes on $C_{c}^{\infty}(U,\fk) \times \fk$ and on 
$\fk \times C_{c}^{\infty}(U,\fk)$,
and factors through a cocycle on $\fk\times\fk$, which is a coboundary by
Whitehead's Lemma. Thus $H^{2}_{\op{ct}}(i)$ is also injective.
\end{proof}

Summarising, we have defined a morphism
$\mu \from \cF_{\op{ct}} \to \cS_{\op{ct}}$ from a sheaf to a
monopresheaf that is an isomorphism on sufficiently small subsets $U$ of
$M$, at least in the case that the fibres of $\fK$ are semi-simple.
According to the following standard proposition, the monopresheaf
$\cS_{\op{ct}}$ must then in fact be a sheaf, and the morphism $\mu$ an
isomorphism of sheaves. This means that in particular
$\mu_M \from \cF_{\op{ct}}(M) \to \cS_{\op{ct}}(M)$ is an isomorphism.
In other words, the map $\phi \mapsto [\phi \circ \omega]$ is an
isomorphism
\begin{equation}\label{eqn:proved_isomorphism}
 \Hom_{\op{ct}}(\ol{\Omega}{}^{1}_{c}(M,V(\fK)),X)
 \cong H^{2}_{\op{ct}}(\Gamma_{c}(\fK),X)\,,
\end{equation}
proving Proposition \ref{prop:universal_algebra_cocycle}.

\begin{proposition}\label{prop:local-iso}
 Let $\cF$ be a sheaf, $\cS$ a monopresheaf (i.e.\  a presheaf that
 satisfies the local identity axiom), and let $\mu \from \cF \to \cS$ be
 a morphism of presheaves such that each $x \in M$ has an open
 neighbourhood $V$ such that $\mu_W \from \cF(W) \to \cS(W)$ is an
 isomorphism for any open $W \subseteq V$. Then $\cS$ is a sheaf, and
 $\mu$ an isomorphism.
\end{proposition}
\begin{proof}
 We show that $\mu_{U} \from \cF(U)\to\cS(U)$ is an isomorphism for
 arbitrary open $U\subseteq M$. Fix an open cover $\{V_i\}_{i \in I}$ of
 $U$ such that $\mu_{W}\from \cF(W)\to\cS(W)$ is an isomorphism for all
 $W \subseteq V_i$. First of all, we show that $\mu_{U}$ is injective.
 Suppose that $\mu_U (f_U) = 0$ in $\cS(U)$. Then certainly
 $\rho_{V_iU}\mu_U (f_U) = \mu_{V_i}\rho_{V_iU} (f_U) =0$ for all
 $i\in I$, and since $\mu_{V_i}$ is an isomorphism we have
 $f_{V_i} := \rho_{V_iU} (f_U) = 0$. But if $f_{V_i}=0$ for all
 $i\in I$, then $f_U$ must be $0$ by the `local identity' axiom for
 $\cF$. Next, we show that $\mu_U$ is surjective. Given
 $s_U \in \cS(U)$, we construct an $f_U \in \cF(U)$ such that
 $\mu_U(f_U) = s_U$. Set $s_{i} := \rho_{V_iU} (s_U)$, so
 $\rho_{V_{ij}V_i}(s_{i}) = \rho_{V_{ij}V_j}(s_j)$ by the presheaf
 property of $\cS$. (We write $V_{ij} = V_i \cap V_j$.) Set
 $f_i := \mu^{-1}_{V_i}(s_i)$ and observe
 $\mu_{V_{ij}}\rho_{V_{ij}V_i}(f_i) = \rho_{V_{ij}V_i}(s_i) = \rho_{V_{ij}V_j}(s_j) = \mu_{V_{ij}}\rho_{V_{ij}V_j}(f_j)$.
 Since $\mu_{V_{ij}}$ is an isomorphism,
 $\rho_{V_{ij}V_i}f_i = \rho_{V_{ij}V_j}f_j$. By the gluing property of
 $\cF$, the $f_i$ then extend to an $f_{U}\in \cF(U)$ with
 $\rho_{V_iU}f_U = f_i$. Since
 $\rho_{V_iU} \mu_{U}(f_U) = \mu_{V_i}\rho_{V_iU}(f_U) = s_i$ for all
 $i \in I$, we must have $\mu_U(f_U) = s_U$ by the `local identity'
 axiom for $\cS$.
\end{proof}

\begin{tabsection}
 Note that the last proposition not only yields Proposition
 \ref{prop:universal_algebra_cocycle} by eqn.\
 \eqref{eqn:proved_isomorphism}, but it also shows that
 $\cS_{\op{ct}}(U) = H^{2}_{\op{ct}}(\Gamma_{c}(\fK|_{U}) , X)$ actually
 constitutes a sheaf.
\end{tabsection}

\begin{remark}
 Another situation where the preceding argument would apply is the case
 $\fk=\frak{pu}(\mathcal{H})$, for $\mathcal{H}$ a separable
 infinite-dimensional Hilbert space. If we endow $PU(\mathcal{H})$ with
 the norm topology, then it becomes an infinite-dimensional Lie group
 with Lie algebra $\frak{pu}(\mathcal{H})$ and equivalence classes of
 smooth principal $PU(\mathcal{H})$-bundles are in bijection with
 $H^{3}(M,\Z)$ for each finite-dimensional manifold $M$
 \cite{MullerWockel07Equivalences-of-Smooth-and-Continuous-Principal-Bundles-with-Infinite-Dimensional-Structure-Group}.
 Picking such a bundle $P\to M$ we consider the associated Lie algebra
 bundle $\fK\to M$. Since $PU(\mathcal{H})$ also acts on the central
 extension $\wh{\fk}=\frak{u}(\mathcal{H})$ we get a central extension
 $\hat{\fK}\to \fK$ of Lie algebra bundles and one of Lie algebras
 \begin{equation*}
  \Gamma_{c}(\hat{\fK}) \to \Gamma_{c}(\fK).
 \end{equation*}
 As above, Proposition \ref{prop:local-iso} would yield the universality
 of this central extension if one knew that this were the case for
 \emph{trivial bundles}, i.e., for the associated current algebras. A
 particularly interesting example would be $M=G$ a simple, 1-connected and
 compact Lie group and
 \begin{equation*}
  [P]\in \op{Bun}(G,PU(\mathcal{H}))\cong H^{3}(G,\Z)\cong \Z
 \end{equation*}
 a generator (cf.\
 \cite{NikolausSachseWockel11A-Smooth-Model-for-the-String-Group}).
\end{remark}

\section{Universality of the Gauge Group Extension}
\label{sect:universality_of_the_gauge_group_extension}

\begin{tabsection}
 In this section, we will use the following Recognition Theorem from
 \cite{Neeb02Universal-central-extensions-of-Lie-groups} in order to
 prove Theorem~\ref{thm:universal_group_extension}.
\end{tabsection}

\begin{theorem}
 Let $Z\to \wh{H}\to H$ be a central extension of Lie groups such that
 $H$ is a connected and locally convex Lie group with perfect Lie algebra $ \fh$, 
 $\fz:=L(Z)$ is Mackey
 complete,
 \begin{enumerate}
  \item the induced Lie algebra extension $\mf{z}\to \wh{\fh}\to \fh$ is
        $\R$-universal and
  \item $\wh{H}$ is simply connected.
 \end{enumerate}
 If the derived extension $\fz\to\wh{\fh}\to\fh$ is universal for a
 Mackey complete space $\mf{a}$, then $\wh{H}$
 is universal for each regular abelian Lie group $A$ with
 $L(A)=\mf{a}$.\vspace{-0.5ex}
\end{theorem}
\begin{proof}
 This is \cite[Th.\
 4.13]{Neeb02Universal-central-extensions-of-Lie-groups}, boiled-down to
 the case that $\fh$ is perfect. The condition $\pi_{1}(H)\se D(\wt{H})$
 of \cite[Th.\ 4.13]{Neeb02Universal-central-extensions-of-Lie-groups}
 is automatically satisfied since in this case $D(\wt{H})$ coincides
 with the universal covering group of $H$ by definition. Moreover, by
 \cite[Lem.\ 4.5]{Neeb02Universal-central-extensions-of-Lie-groups}
 universality and weakly universality of Lie algebra extensions are
 equivalent. The conclusion of \cite[Th.\
 4.13]{Neeb02Universal-central-extensions-of-Lie-groups} can be
 strengthened to yield the universality of $Z\to \wh{H}\to H$ by
 \cite[Lem.\ 4.5]{Neeb02Universal-central-extensions-of-Lie-groups}.
\end{proof}
\begin{proof}
 (of Theorem \ref{thm:universal_group_extension}) Proposition
 \ref{prop:universal_algebra_cocycle} and \cite[Lem.\
 1.12]{Neeb02Universal-central-extensions-of-Lie-groups} show that the
 induced Lie algebra extension \eqref{eqn:universal_algebra_cocycle} is
 universal (even in the stronger sense of
 \cite{Neeb02Universal-central-extensions-of-Lie-groups}). By the
 Recognition Theorem it thus remains to check that
 $\wh{\Gamma(\cK)_{0}}$ is simply connected. We first note that by
 \cite[Rem.\
 7.14]{Neeb02Central-extensions-of-infinite-dimensional-Lie-groups} we
 have
 $Z\cong \ol{\Omega}{}^{1}_{c}(M,V(\fK))/\Pi_{\omega}\times \pi_{1}(\Gamma(\cK)_{0})$.
 If we consider the long exact homotopy sequence
 \begin{equation*}
  \pi_{2}(\Gamma(\cK)_{0})\xrightarrow{\delta_{2}} \pi_{1}(Z)\to \pi_{1}(\wh{\Gamma(\cK)_{0}})\to \pi_{1}({\Gamma(\cK)_{0}})\xrightarrow{\delta_{1}} \pi_{0}(Z),
 \end{equation*}
 then $\delta_{2}$ is surjective by \cite[Prop.\
 5.11]{Neeb02Central-extensions-of-infinite-dimensional-Lie-groups} and
 $\delta_{1}$ is an isomorphism by construction. Thus
 $\wh{\Gamma(\cK)_{0}}$ is simply connected.
\end{proof}

\section*{Acknowledgements}

\begin{tabsection}
 B.J.\ would like to thank Jeroen Sijsling for some illuminating
 comments, and we also thank Karl-Hermann Neeb for various useful comments 
on an earlier version of the manuscript.
 Parts of the work on this paper were carried out while B.J.\
 enjoyed a fellowship from the Collaborative Research Center SFB 676
 \emph{Particles, Strings, and the Early Universe}.
\end{tabsection}

\def\polhk#1{\setbox0=\hbox{#1}{\ooalign{\hidewidth
  \lower1.5ex\hbox{`}\hidewidth\crcr\unhbox0}}} \def\cprime{$'$}

{}

\end{document}